\documentclass[12pt,reqno]{amsart}

\usepackage{amsmath}
\usepackage{amscd,amssymb}
\usepackage{epsfig}

\usepackage{graphicx}
\usepackage{color}
\usepackage{hyperref}
\usepackage{booktabs}
\usepackage{amsfonts}
\usepackage{mathrsfs}
\usepackage{bm}
\usepackage{fancyhdr}
\usepackage{amsthm}
\usepackage{txfonts}
\usepackage{float}
\usepackage{multirow}

\def\qed{\ifmmode\square\else\nolinebreak\hfill
$\Box$\fi\par\vskip12pt}

\newtheorem{thm}{Theorem}[section]
\newtheorem{lemma}[thm]{Lemma}
\newtheorem{corollary}[thm]{Corollary}
\newtheorem{proposition}[thm]{Proposition}
\numberwithin{equation}{section}
\numberwithin{thm}{section}

\theoremstyle{definition}

\newtheorem{remark}[thm]{Remark}
\newtheorem{example}[thm]{Example}

\newcommand{\bF}{\mathbb F}

\newcommand{\bZ}{\mathbb Z}

\definecolor{Purple}{rgb}{0.5,0,0.5}

\renewcommand{\>}{\rangle}

\begin{document}
\pagestyle{plain}
\parindent=19pt
\begin{titlepage}

\title{New constructions of signed difference sets}

\begin{center}
\author{Zhiwen He, Tingting Chen$^{*}$ and Gennian Ge
}\end{center}
\address{Zhejiang Lab, Hangzhou 311100, China}
\email{ zhiwenhe94@163.com }

\address{Zhejiang Lab, Hangzhou 311100, China}
\email{ ttchenxu@mail.ustc.edu.cn }

\address{School of Mathematics Sciences, Capital Normal University, Beijing 100048, China}
\email{ gnge@zju.edu.cn }
%%%%%%%%%%%%%%%%%%%%%%%%%%%%%%%%%%%
\begin{abstract}
    Signed difference sets have interesting applications in communications and coding theory. A $(v,k,\lambda)$-difference set in a finite group $G$ of order $v$ is a subset $D$ of $G$ with $k$ distinct elements such that the expressions $xy^{-1}$ for all distinct two elements $x,y\in D$, represent each non-identity element in $G$ exactly $\lambda$ times. A $(v,k,\lambda)$-signed difference set is a generalization of a $(v,k,\lambda)$-difference set $D$, which satisfies all properties of $D$, but has a sign for each element in $D$. We will show some new existence results for signed difference sets by using partial difference sets, product methods, and cyclotomic classes.
\end{abstract}


\maketitle

\section{Introduction}

Difference sets and their generalizations have vast applications  in communication and radar
systems \cite{LiuS2018}, cryptography \cite{Cusick2015} and coding theory \cite{Ding}. In particular, spectrally constrained sequences with low cross-correlations can be used in cognitive radar \cite{Tsai2011}, and the spectral null constraints can be supported by difference sets \cite{YeZ2022}.

Signed difference set is a generalization of the difference set, which was first introduced by Gordon \cite{Gordon}. The formal definition is as follows. Let $G$ be a finite group of order $v$ and $\bZ[G]$ denote the group ring of $G$ over $\bZ$. Let $D=\sum_{i=1}^ks_id_i\in \bZ[G]$ with $s_i\in \{-1,1\}$, $1\leq i\leq k$. If the signed set $D$ satisfies 
\begin{equation}\label{Eqn_DS}
DD^{-1}=\lambda G+n\cdot 0_G,
\end{equation}
where $n=k-\lambda$, we call $D$ a $(v,k,\lambda)${\it -signed difference set} (SDS). Denote $P,N$ as sets such that $D=P-N$, i.e., $P=\{d_i:s_i=1\text{ for all } i\in[k]\}$ and $N=\{d_i:s_i=-1\text{ for all } i\in[k]\}$. If $N=\emptyset$, the subset $D$ that satisfies Eq. ~\ref{Eqn_DS} is a $(v,k,\lambda)${\it -difference set} (DS). If $N=\emptyset$ and $D$ satisfies 
\begin{equation}\label{Eqn_PDS}
DD^{-1}=\lambda D+\mu(G-D-1_G)+k\cdot 0_G,
\end{equation} 
we call $D$ a {\it $(v,k,\lambda,\mu)$-partial difference set} (PDS). In particular, when the identity $0_G\notin D$ and $D^{(-1)}=D$, we call $D$ is {\it regular}. A {\it Paley PDS} is a PDS with parameters $(v,\frac{v-1}{2},\frac{v-5}{4},\frac{v-1}{4})$. It is easy to see that when $\lambda=\mu$, $D$ is exactly a $(v,k,\lambda)$-DS. PDSs and DSs have close connections with combinatorial structures, as well as sequences and codes. See \cite{Tao,Ding,Ye} for instance. Their research is more advanced than that on SDSs, as demonstrated by notable papers such as \cite{Hall} by Hall and \cite{S.L.MA} by Ma.

An SDS with $\lambda=0$ is the concept that we call weighing matrix. A {\it weighing matrix $W(v,k)$} is a $v\times v$ matrix with entries in $\{-1,0,1\}$ such that 
\[
WW^T=kI,
\]
for some positive integer $k$, where $W^T$ is the transpose of $W$ and $I$ is the identity matrix with suitable size. Let $G$ be a group of order $v$. If $W=(W_{g_i,g_j})_{0\leq i,j\leq v-1}$ indexed with the elements of $G=\{g_0,g_1,...,g_{v-1}\}$ and $W_{g_ig,g_jg}=W_{g_i,g_j}$ for any $g_i,g_j,g\in G$, $W$ is said to be {\it $G$-invariant}. Note that a $G$-invariant matrix can be identified solely by its first row. When $G$ is a cyclic group, we call $W$ a {\it circulant} weighing matrix. When $\lambda=0$, a $(v,k,0)$-SDS corresponds to a $G$-invariant weighing matrix $W(v,k)$ via the formula below, and vice versa.
\begin{equation}\label{Eqn_W1j}
W_{g_0,g}=\begin{cases}1&\textup{ if }g\in P,\\
0&\textup{ if }g\notin P\cup N,\\
-1&\textup{ if }g\in N.
\end{cases}
\end{equation}
The existence of group invariant weighing matrices has been extensively studied, see for example \cite{ArasuKT2,Tan,LeKH}.

Consider an SDS $D=P-N$ with parameters $(v,k,\lambda)$ over a group $G$, then it is easy to see that $-D=N-P$ is also a $(v,k,\lambda)$-SDS.  Without loss of generality, we consider $D$ and $-D$ of the same type. By counting all differences of positive and negative signs in Eq. ~\ref{Eqn_DS}, we can get that 
\begin{align}\label{ness1}
    |P|(|P|-1)+|N|(|N|-1)-2|P|\cdot|N|=(|P|-|N|)^2-(|P|+|N|)=\lambda(v-1).
\end{align}
Eq.~\ref{ness1} implies a necessary condition for SDS that $\lambda\geq-1$ and $k+\lambda(v-1)=(|P|-|N|)^2$ must be a perfect square, which is similar to weighing matrices that are requested for $k=(|P|-|N|)^2$ to be a perfect square. The websites in \cite{D.M,A.E} contain online databases with the existence results and parameters for PDSs, DSs, circulant weighing matrices and SDSs. Constructing SDSs with new parameters on different group structures to enrich the parameter sets would be an interesting endeavor.

% If we use the same definition as $W_{g_0,g}$ in Eq. ~\ref{Eqn_W1j}, then an SDS $D$ defined on a cyclic group $G=\{0,1,...,v-1\}$ can also correspond to a ternary periodic sequence $S_D$ with 
%\begin{equation}\label{Eqn_S_D}
%S_D(i)=W_{0,i}, 0\leq i\leq v-1. 
%\end{equation}

Using the similar notations in Eq. ~\ref{Eqn_W1j}, we can define a ternary sequence $S_D$ of period $v$ from a $(v,k,\lambda)$-SDS $D=P-N$ as follows: $S_D=(s_0,s_1,\ldots,s_{v-1})$, where for each $i\in[0,v-1]$ $s_i=1,-1,0$ if $g_i\in P,N,G-(P\cup N)$ respectively. The {\it periodic autocorrelation $C_S(\tau)$} of $S_D$ at shift $\tau$ is defined by 
\[
C_S(\tau)=\sum_{i=0}^{v-1}s_{i+\tau}s_i^*,\ 0\leq \tau<v,
\]
where $s_i^*$ is the complex conjugate of $s_i$ and $i+\tau$ is reduced by modulo $v$. A sequence $S$ of length $v$ is called a {\it two-level sequence} if its all out-of-phase periodic autocorrelation values equal $-1$, i.e., $C_S(\tau)=-1$ for any $0<\tau<v$. The search for new sequences with two-level autocorrelation has been a fascinating problem for decades. Due to their good correlation properties, these sequences are widely used in communications and cryptography. See \cite{Helleseth,Xiang,Wang} and references therein. In fact, the autocorrelation values  $C_S(\tau)$ are associated with the parameter $\lambda$, we will show this in the next sections.

In \cite{Gordon}, Gordon gave some constructions for SDSs from DSs and PDSs. By utilizing quadratic residues and characters, Gordon gave two creative constructions. And in \cite{Jose}, the authors investigated the construction of DSs in the additive group over a finite field by using the tools of cyclotomic classes.  Motivated by their methods, in this work, we construct SDSs by Paley PDSs and cyclotomic classes over finite fields. Our main contributions are listed below.

\begin{itemize}
    \item[(1)] We show that the signed set $D$ is an SDS with $\lambda=-1$ if and only if $S_D$ has two-level autocorrelations.
    \item[(2)] Consider $v$ as an odd integer. By using PDSs, we give some constructions for SDSs with parameters $(v,v-1,-1)$ and $(243,242,161)$, where $v$ is a prime power and $v\equiv 1\pmod{4}$  or $v=n^4$ or $9n^4$ for $n$ being an odd integer. These parameters are more general than those in the construction of $(v,v-1,-1)$-SDS given by Gordon, which needs that $v$ is an odd prime number.
    \item[(3)] By using product construction and character theory on $3$-groups, we obtain an infinite class of SDSs with parameters $(3^{2m+1},3^{2m}+1,1)$ for any positive even integer $m$.
    \item[(4)] We classify all fourth-order cyclotomic SDSs, and obtain  $(v,k,\lambda)$-SDS with $\lambda=v-4$, $\frac{v-9}{4}$, $\frac{v-1}{4},-1$, $\frac{f-1}{4}$, $\frac{f-5}{4}$, $\frac{9f-9}{4}$, $\frac{f-3}{4}$. In particular,  when $k=f+1=\frac{v+3}{4}$ and $\lambda=\frac{f-3}{4}=\frac{v-13}{16}$, we can give an alternative proof of \cite[Theorem 3.2]{Gordon}.
\end{itemize}

The rest of this paper is organized as follows. In Section~\ref{sec_pre}, we give some definitions, notations, and results on group rings, character theory, and sequences. In Section~\ref{Sec_PDS}, we construct SDSs by using PDSs. Based on these SDSs, we give a construction of SDSs by using product methods and character theory in Section~\ref{sec_3group}. Details of constructions of fourth-order cyclotomic SDSs are given in Section~\ref{sec_cyc}. Finally, a brief conclusion is given in Section~\ref{sec_con}.
\end{titlepage}

\section{Preliminaries}\label{sec_pre}
In this section, we will give some definitions and notations for group rings and character theory. We also give the connections between SDSs and two-level autocorrelation sequences.

\subsection{Group rings and characters}

Let $G$ be a finite group. The group ring $\bZ[G]$ is defined as a collection of formal sums, where the elements are taken from $G$ and the coefficients are integers from $\bZ$. The operations ``$+$'' and ``$\cdot$'' of $\bZ[G]$ are given by 
\begin{equation}
\sum_{g\in G}a_g g+\sum_{g\in G}b_g g=\sum_{g\in G}(a_g+b_g)g
\end{equation} 
and 
\begin{equation}
(\sum_{g\in G}a_g g)\cdot(\sum_{h \in G}b_h h)=\sum_{g,h\in G}a_gb_h(gh).
\end{equation} 
Numerous combinatorial structures are extensively explored within the framework of a group ring. It becomes conventional to abuse the notation $S$ as a subset of $G$ and the corresponding element $\sum_{s\in S}s$ in $\bZ[G]$ at the same time. Furthermore, within the following context, we will explore character theory, a powerful tool that can greatly simplify calculations. A character of a finite abelian group is defined as a homomorphism from $G$ to the multiplicative group of complex numbers with absolute value of $1$. The principal character, denoted as  $\chi_0$, is such that $\chi_0(x)=1$ for any $x\in G$. For a subset $S$ of $G$, we write $\chi(S)=\sum_{s\in S}\chi(s)$. 

We need the following well-known inverse formula on group rings:
\begin{lemma}\label{Eqn_chi}
Let $A=\sum_{g\in G}a_gg$ be an element of the group ring $\bZ[G]$ for a finite abelian group $G$. Then the coefficients $a_g$'s of $A$ can be computed explicitly by
\[
a_g=\frac{1}{|G|}\sum_{\chi\in\hat{G}}\chi(A)\chi(g^{-1}),
\]
where $\hat{G}$ denotes the character group of $G$. In particular, if $A,B\in \bZ[G]$ satisfy $\chi(A)=\chi(B)$ for all characters $\chi\in\hat{G}$, then $A=B$.
\end{lemma}

By Lemma \ref{Eqn_chi} and the definition of SDS, we get the following result.

\begin{lemma}\label{lem_charSDS}
Let $G$ be a group of order $v$ and $D\in \bZ[G]$ be a signed subset of $G$ with size $k$. Then $D$ is an SDS of $G$ if and only if $\chi(D)\overline{\chi(D)}=n$ for any non-principal character $\chi\in\hat{G}$.  
\end{lemma}

Let $\bF_q$ be a finite field, where $q=p^m$ and $p$ is a prime. Let $w$ be the primitive element of $\bF_q$. Then $\{1,w,...,w^{m-1}\}$ forms a polynomial basis of $\bF_q$ over $\bF_p$. Let $V$ be a vector space of dimension $m$ over $\bF_p$. For each element $x\in \bF_q$, it can be written 
uniquely as $x=\sum_{i=0}^{m-1}x_iw^i$ for some vector $(x_0,x_1,...,x_{m-1})\in V$. Let $G=(\bF_q,+)$ be the additive group of $\bF_q$ and $\hat{G}$ denote the additive character group of $G$. It is easy to see that $\hat{G}=\{\psi_a: a=\sum_{i=0}^{m-1}a_iw^i\in \bF_q\}$ with each character $\psi_a$ being defined as follows
\begin{equation}\label{Eqn_char}
\psi_a:x\mapsto\xi_p^{a\cdot x}, x=(x_0,x_1,...,x_{m-1})\in \bF_q,
\end{equation}
where $a\cdot x=\sum_{i=0}^{m-1}a_ix_i\in\bF_p$, and $\xi_p$ is a complex primitive $p$-th root of unity.

\subsection{Two-level autocorrelation sequences}
Sequences with low autocorrelation have vast applications in communication systems. Ternary two-level autocorrelation sequences are strongly related to SDSs. Given a ternary two-level autocorrelation sequence of period $v$ defined from an SDS, its all  autocorrelations are totally determined by the parameter $\lambda$.

\begin{lemma}
Let $G=\bZ_v$ be a cyclic group of order $v$ and $D\in\bZ[G]$ be a signed set of $G$. Let $S_D$ be the ternary sequence defined from $D$. Then the signed set $D$ is an SDS with $\lambda=-1$ if and only if $S_D$ has two-level autocorrelation.
\begin{proof}
Let $P$ and $N$ be two subsets of $G$ such that $D=P-N$. By the definition of the sequence $S_D$, for any $0\leq \tau<v$, we have
\[
\begin{aligned}
C_S(\tau)&=\sum_{i=0}^{v-1}s_{i+\tau}s_i^*\\
&=|P\cap P+\tau|+|N\cap N+\tau|-|P\cap N+\tau|-|N\cap P+\tau|\\
&=DD^{-1}(\tau),
\end{aligned}
\]
where $DD^{-1}(\tau)$ denotes the coefficient of the element $\tau$ in $DD^{-1}$. Note that $D$ is an SDS with $\lambda=-1$ if and only if $DD^{-1}=-(G-0_G)+k\cdot 0_G$. This completes the proof.
\end{proof}
\end{lemma}

\section{Constructions by Paley type partial difference sets}\label{Sec_PDS}
In this section, we are going to construct SDSs using PDSs. Under our construction model, PDSs need to satisfy the condition $\lambda-\mu=-1$. This type of PDS is either the Paley PDS or the PDS with parameters $(243,22,1,2)$ or $(243, 220, 199, 200)$. By applying these PDSs to our construction, we obtain a lot of SDSs with new parameters.
 
\begin{lemma}\label{sds_pds}
    Let G be a group of order $v$ and $D'$ be a regular $(v,k,\lambda,\mu)$-PDS. If $\lambda-\mu=-1$, then there exists an SDS with parameters $(v,v-1,v-4k+4\mu-2)$.
\end{lemma}
\begin{proof}
    Let $P=D'$, $N=G-P-0_G$ and $D=P-N$. Then we get that
    \begin{equation}\label{Eqn_SDS_PDS1}
    \begin{aligned}
        DD^{-1}&=(2D'-G+0_G)(2D'-G+0_G)^{-1}\\
        &=(|G|-4|P|-2)G+4D'D'^{-1}+2D'+2D'^{-1}+0_G.
    \end{aligned}
    \end{equation}
    Since $D'$ is a PDS with $\lambda-\mu=-1$, we have
    \begin{equation}\label{Eqn_SDS_PDS2}
    \begin{aligned}
    D'D'^{-1}&=\lambda D'+\mu(G-D'-0_G)+k0_G\\
    &=-D'+\mu(G-0_G)+k0_G.
    \end{aligned}
    \end{equation}
    Substitute Eq. ~\ref{Eqn_SDS_PDS2} into Eq. ~\ref{Eqn_SDS_PDS1},
    \[
     DD^{-1}=(v-4k+4\mu-2)G+(4k-4\mu+1)0_G.
    \]
    Hence $D$ is an SDS with parameters $(v,v-1,v-4k+4\mu-2)$.
\end{proof}

Now we consider the PDSs with $\lambda-\mu=-1$. The next result in \cite{ArasuKT} tells us that for certain classes of regular abelian PDSs, almost all are Paley PDSs.

%\begin{thm}\label{Eqn_SDS_PDS}
%Let $G$ be a group of order $v$, and $P$ is a regular PDS with parameters $(v,k,\lambda,\mu)$. If $\lambda-\mu=-1$, then there exists an SDS with parameters $(v,v-1,v-4k+4\mu-2)$.
%\end{thm}
%\begin{proof}
%Let $N=G-P-1_G$ and $D=P-N$. Then we get that 
%\begin{equation}\label{Eqn_prop_PDS1}
%\begin{aligned}
%DD^{-1}&=(2P-G+1_G)(2P-G+1_G)^{-1}\\
%&=(|G|-4|P|-2)G+4PP^{-1}+2P+2P^{-1}+1_G.
%\end{aligned}
%\end{equation} 
%Since $P$ is a PDS with $\lambda-\mu=-1$, we have 
%\begin{equation}\label{Eqn_prop_PDS2}
%\begin{aligned}
%PP^{-1}&=\lambda P+\mu(G-P-1_G)+k\cdot 1_G\\
%&=-P+\mu (G-1_G)+k 1_G.
%\end{aligned}
%\end{equation}
%Substitute \ref{Eqn_prop_PDS2} into \ref{Eqn_prop_PDS1}, 
%\[
%DD^{-1}=(v-4k+4\mu-2)G+(4k-4\mu+1)1_G.
%\]
%Hence $D$ is an SDS with parameters $(v,v-1,v-4k+4\mu-2)$.
%\end{proof}

\begin{lemma}
Suppose that there exists a nontrivial regular abelian $(v,k,\lambda,\mu)$-PDS with $\lambda-\mu=-1$. Then either $D$ is a Paley PDS or $(v,k,\lambda,\mu)=(243,22,1,2)$ or $(243,220,199,200)$.
\end{lemma}

Paley PDSs are of interest in design theory as they can be used to construct some special types of divisible designs and divisible difference sets. See \cite{ArasuKT,ArasuKT1} for instances. Also, Paley PDSs in abelian groups can be used to construct binary sequences and arrays with small periodic off-phase auto-correlation, and hence have applications in communication science \cite{ChanYK}. Below are some results on the existence of Paley PDSs.

\begin{lemma}[\cite{S.L.MA1}]\label{lem_sl}
Let $G$ be a cyclic group of order $v\equiv 1\pmod{4}$, then $G$ contains a Paley PDS $D$ if and only if $v$ is a prime number and $D$ is the set of quadratic residues modulo $p$. 
\end{lemma}

\begin{lemma}[\cite{Polhill}]\label{lem_pol}
Let $n>1$ be an odd integer. Then there exists a Paley PDS in a group of orders $n^4$ and $9n^4$.
\end{lemma}

\begin{lemma}[\cite{WangZ}]\label{lem_wz}
Let $v>1$ be an odd integer and $G$ be an abelian group of order $v$. Then $G$ contains a Paley PDS if and only if $v$ is a prime power and $v\equiv1\pmod{4}$, or $v=n^4$ or $9n^4$ with $n>1$ being an odd integer. 
\end{lemma}

\begin{remark}
In \cite{Golay}, Golay constructed a perfect two error correcting $(11,6)$-linear code $C$ over $\bF_3$. The dual code $C^{\perp}$ is a two weight $(11,5)$-projective code with weights $6$ and $9$. By \cite[Theorem 8.1]{S.L.MA}, we can get a regular $(243,22,1,2)$-PDS $D$ in the additive group $G=(\bF_3^5,+)$. Note that $G-D-0_G$ is also a PDS with parameters $(243,220,199,200)$.
\end{remark}

%\begin{lemma}\label{Eqn_Paley}
%Let $G$ be a group of order $v$ and $P$ is a Paley type PDS with parameters $(v,\frac{v-1}{2},\frac{v-5}{4},\frac{v-1}{4})$. Then the SDSs stated in Theorem \ref{Eqn_SDS_PDS} have parameters $(v,v-1,-1)$.
%\end{lemma}

\begin{thm}\label{Eqn_Paley}
Let $G$ be a group of order $v$ and $D'$ be a Paley type PDS with parameters $(v,\frac{v-1}{2},\frac{v-5}{4},\frac{v-1}{4})$. Then there exists an SDS $D$ with parameters $(v,v-1,-1)$.
\begin{proof}
Let $P=D'$, $N=G-P-0_G$ and $D=P-N$. Since $\lambda-\mu=-1$, we get that $D$ is an SDS of size $v-1$ and $\lambda=-1$ from Lemma \ref{sds_pds}. 
\end{proof}
\end{thm}

By Lemmas \ref{lem_sl}-\ref{lem_wz} and Theorem \ref{Eqn_Paley}, we can easily get the following corollaries.

\begin{corollary}
Let $n$ be an odd integer and $v$ be an integer satisfying that $v=n^4$ or $v=9n^4$. Suppose that $G$ is a group of order $v$. Then $G$ contains SDSs with parameters $(n^4,n^4-1,-1)$ and $(9n^4,9n^4-1,-1)$.
\end{corollary}

\begin{corollary}
Let $p$ be a prime, $v=p^n$ and $v\equiv 1\pmod{4}$. Let $G$ be an abelian group of order $v$. Then $G$ contains SDSs with parameters $(v,v-1,-1)$.
\end{corollary}

\begin{corollary}
Let $G$ be an elementary abelian $3$-group of order $243$. Then there exists an SDS with parameters $(243,242,161)$.
\end{corollary}

The following example is a nontrivial SDS constructed from a regular Paley PDS given in \cite{Paley}.

\begin{example}\label{exm_Paley} 
Let $G=(\bF_q,+)$ be the additive group of a finite field $\bF_q$, where $q$ is an odd prime power and $q\equiv 1\pmod{4}$. The set $D'$ of all nonzero squares in $\bF_q$ forms a Paley PDS. Define the signed set $D=P-N$ with $P=D'$ and $N=G-P-0_G$. Then $D$ is a $(q,q-1,-1)$-SDS.
\end{example}

\section{Product constructions in $3$-groups}\label{sec_3group}
In this section, we construct SDSs on the $3$-groups. Using the method of product construction and character theory, we obtain an infinite class of SDSs with parameters $(3^{2m+1},3^{2m}+1,1)$ based on the SDSs obtained in Section \ref{Sec_PDS}. In this class of SDSs, neither $P$ nor $N$ is a DS or PDS.

Let $m$ be a positive even integer and $G=(\bF_{3^{2m+1}},+)$ be the additive group of finite field $\bF_{3^{2m+1}}$. Let $V$ be a vector space of dimension $2m+1$ over $\bF_3$. We take subspaces $W_0,W_1$ of $V$ such that $V\cong W_0\oplus W_1\oplus W_1$, where $W_0, W_1$ have dimension $1$ and $m$, respectively. Let $G_0=(W_0,+)$ and $G_1=(W_1,+)$. Note that $G\cong G_0\times G_1\times G_1$. Let $P'$ be a Paley PDS in $G_1$ and $N'=G_1-P'-0_{G_1}$. Then $D'=P'-N'$ is a $(3^m,3^m-1,-1)$-SDS by Theorem \ref{Eqn_Paley}. We take $x_0\in G_0\setminus \{0_{G_0}\}$, $x_1\in G_1$ and define the subsets $P$ and $N$ of $G$ and signed set $D$ as follows:
\begin{equation}\label{Eqn_3SDS} 
\begin{aligned}
P&=(x_0,0_{G_1},G_1)+(0_{G_0},G_1,x_1)+(0_{G_0},P',P')+(0_{G_0},N',N'),\\
N&=(0_{G_0},P',N')+(0_{G_0},N',P'),\\
D&=P-N=(x_0,0_{G_1},G_1)+(0_{G_0},G_1,x_1)+(0_{G_0},D',D').
\end{aligned}
\end{equation}
The following theorem tells us that $D$ is an SDS.
%\begin{equation}\label{Eqn_3SDS}
%D=P-N=(x_0,0_{G_1},G_1)\cup(0_{G_0},G_1,x_1)\cup(0_{G_0},D',D').
%\end{equation}

\begin{thm}\label{thm_3SDS}
Let $m$ be an even integer and $\bF_{3^{2m+1}}$ be a finite field.  There exist SDSs with parameters $(3^{2m+1},3^{2m}+1,1)$.
\end{thm}
\begin{proof}
We will use the notations defined as above and show that the signed set $D$ listed in Eq. ~\ref{Eqn_3SDS} is an SDS. By the structure of $D$, we can obtain that $|D|=|G_1|+|G_1|+|D'|\times |D'|=3^{2m}+1$. To see $D$ being an SDS with $\lambda=1$, by Lemma \ref{lem_charSDS},  it suffices to prove that for any non-principal character $\chi\in \hat{G}$,
\[
\chi(D)\overline{\chi(D)}=3^{2m}=k-\lambda=|D|-1.
\] 
For each character $\chi\in\hat{G}$, there exist $\phi_a\in\hat{G_0}$ and $\psi_{b_1},\psi_{b_2}\in\hat{G_1}$, where $\phi_a,\psi_{b_1},\psi_{b_2}$ are as forms of Eq. ~\ref{Eqn_char}. Hence we have 
\[
\hat{G}=\{\chi_{a,b_1,b_2}=\phi_a\psi_{b_1}\psi_{b_2}:a\in G_0,b_1,b_2\in G_1\}.
\] 
For any $\chi_{a,b_1,b_2}\in \hat{G}$, 
\[
\begin{aligned}
\chi_{a,b_1,b_2}(D)&=\phi_a(x_0)\psi_{b_1}(0_{G_1})\psi_{b_2}(G_1)+\phi_a(0_{G_0})\psi_{b_1}(G_1)\psi_{b_2}(x_1)+\phi_a(0_{G_0})\psi_{b_1}(D')\psi_{b_2}(D')\\
&=\phi_a(x_0)\psi_{b_2}(G_1)+\psi_{b_1}(G_1)\psi_{b_2}(x_1)+\psi_{b_1}(D')\psi_{b_2}(D').
\end{aligned}
\]
Note that $\psi_0(D')=|P'|-|N'|=0$ and 
\[
\psi_b(G_1)=\begin{cases}|G_1|,&\textup{ if }b=0_{G_1},\\0,&\textup{ if }b\neq 0_{G_1},\end{cases}
\]
 so we discuss the following cases. 
\begin{itemize}
    \item[(1)]$b_1=0_{G_1}$ and $b_2\neq 0_{G_1}$. In this case, we have
    \[
        \chi_{a,b_1,b_2}(D)\overline{\chi_{a,b_1,b_2}(D)}=|G_1|^2\psi_{b_2}(x_1)\overline{\psi_{b_2}(x_1)}=3^{2m}.
    \]
    \item[(2)]$b_1\neq0_{G_1}$ and $b_2=0_{G_1}$. In this case, we have
    \[
        \chi_{a,b_1,b_2}(D)\overline{\chi_{a,b_1,b_2}(D)}=|G_1|^2\phi_a(x_0)\overline{\phi_a(x_0)}=3^{2m}.
    \]
    \item[(3)]$b_1=0_{G_1}$, $b_2=0_{G_1}$ and $a\neq 0_{G_0}$. Since $a\neq 0_{G_0}$ and $x_0\neq 0_{G_0}$, we have $\phi_a(x_0)=\xi_3$ or $\xi_3^2$.
    \[
        \chi_{a,b_1,b_2}(D)\overline{\chi_{a,b_1,b_2}(D)}=|G_1|^2\xi_3^2\xi_3=3^{2m}.
    \]
    \item[(4)]$b_1\ne 0_{G_1}$ and $b_2\ne 0_{G_1}$. Since $D'$ is a $(3^m,3^m-1,-1)$-SDS, we have $\psi_b(D')\overline{\psi_b(D')}=3^m$ for any non-zero element $b\in G_1$. This implies that 
    \[
        \chi_{a,b_1,b_2}(D)\overline{\chi_{a,b_1,b_2}(D)}=\psi_{b_1}(D')\overline{\psi_{b_1}(D')}\psi_{b_2}(D')\overline{\psi_{b_2}(D')}=3^{2m}.
    \]
\end{itemize}
 Hence $D$ is an SDS with parameters $(3^{2m+1},3^{2m}+1,1)$.
\end{proof}

\begin{example}
Let $G=(\bF_{3^5},+)\cong(\bF_3^5,+)$ and let $G_1=(\bF_3^2,+)$. Take 
\[
\begin{aligned}
    P'&=\{(1,1),(2,1),(1,2),(2,2)\},\\
    N'&=G_1-P'-0_{G_1}.
\end{aligned}
\]
 Then $P'$ is a Paley type PDS with parameters $(9,4,1,2)$ and $D'=P'-N'$ is a $(9,8,-1)$-SDS. The following signed set $D$ is a $(243,82,1)$-SDS: 
\[
D=(1,0_{G_1},G_1)+(0_{G_0},G_1,(1,0))+(0_{G_0},D',D').
\]
\end{example}

\section{Constructions by cyclotomic classes}\label{sec_cyc}
In this section, we will construct an SDS $D=P-N$ over $G=\bF_q$ by using cyclotomic classes. In particular, we consider the set $P,N$ not only to be a single cyclotomic class $A$ or the unions of cyclotomic classes $B$, but also $A\pm 0_G$ or $B\pm 0_G$.

Let $\bF_q$ be a finite field with $q$ elements, where $q$ is a prime power. Consider the multiplicative group $\bF_q^*$. Let $w$ be a primitive element of $\bF_q^*$, and write $q-1=ef$, where $e,f>1$ are two integers. Let 
\begin{equation}\label{Eqn_C0}
C_0=\<w^e\>=\left\{w^{ej}|0\leq j\leq f-1\right\}
\end{equation}
be the subgroup of $\bF_q^*$ generated by $w^e$, and let 
\begin{equation}
C_i=w^iC_0=\left\{w^ix|x\in C_0\right\}, \textup{for} \ 0\leq i\leq e-1,
\end{equation}
be the cyclotomic cosets of $C_0$. Then the cyclotomic numbers are defined by
\begin{equation}
(i,j)_e=|(C_i+1)\cap C_j|
\end{equation}
for $0\leq i,j\leq e-1$. To simplify notations, we omit the subscript $e$ and denote $(i,j)_e$ by $(i,j)$. The partition $C_0,C_1,...,C_{e-1},\{0_G\}$ of $\bF_q$ forms a subalgebra, which is called a Schur ring, of the group ring $\bZ[G]$, where $G=(\bF_q,+)$ is the additive group of $\bF_q$. In other words, the subspace spanned by $C_i$'s and $0_G$ is closed under the multiplication
\begin{equation}
C_iC_j=\left\{x+y|x\in C_i,y\in C_j \right\}=\sum_{k=0}^{e-1}(j-i,k)C_{k+i}+|C_0\cap -C_{j-i}|0_G
\end{equation}
for any $0\leq i,j\leq f-1$. 

A powerful method for constructing SDS in the additive groups of finite fields is employing cyclotomic class. We will classify all fourth-order ($e=4$) cyclotomic SDS using the cyclotomic number table. Before proceeding with the classification, we  list the results of Katre and Rajwade \cite{Katre} regarding fourth-order cyclotomic numbers. 

\begin{thm}\label{thm_cctable}
Let $p$ be an odd prime, $q=p^n\equiv 1\pmod{4}$, $q=1+4f$. Let $\bF_q$ be a finite field and $w$ be a generator of $\bF_q^*$. If $p\equiv-1\pmod{4}$, let $s=(-p)^{n/2}$ and $t=0$. If $p\equiv1 \pmod{4}$, let $s$ be uniquely determined by $q=s^2+t^2$, $p\nmid s$, $s\equiv1\pmod{4}$, and  $t$ be uniquely determined by $w^{(q-1)/4}\equiv s/t\mod{p} $. Then the cyclotomic numbers of order $4$ for $\bF_q$, corresponding to $w$, are determined unambiguously by Tables ~\ref{table 1}-\ref{table 2}.
\end{thm}
\begin{table}[h]
\caption{$f$ even}
\label{table 1}
\centering
\begin{tabular}{|c|c|c|c|c|}\hline
&$0$&$1$&$2$&$3$\\ \hline
$0$&$A$&$B$&$C$&$D$\\ \hline
$1$&$B$&$D$&$E$&$E$\\ \hline
$2$&$C$&$E$&$C$&$E$\\ \hline
$3$&$D$&$E$&$E$&$B$\\ \hline
\end{tabular}
\end{table}
In Table~\ref{table 1},
\begin{equation}
\begin{array}{l}
A=\frac{1}{16}(q-11-6s),\\
B=\frac{1}{16}(q-3+2s+4t),\\
C=\frac{1}{16}(q-3+2s),\\
D=\frac{1}{16}(q-3+2s-4t),\\
E=\frac{1}{16}(q+1-2s).
\end{array}
\end{equation}

\begin{table}[h]
\caption{$f$ odd}
\label{table 2}
\centering
\begin{tabular}{|c|c|c|c|c|}\hline
&$0$&$1$&$2$&$3$\\ \hline
$0$&$A$&$B$&$C$&$D$\\ \hline
$1$&$E$&$E$&$D$&$B$\\ \hline
$2$&$A$&$E$&$A$&$E$\\ \hline
$3$&$E$&$D$&$B$&$E$\\ \hline
\end{tabular}
\end{table}
In Table~\ref{table 2},
\begin{equation}
\begin{array}{l}
A=\frac{1}{16}(q-7+2s),\\
B=\frac{1}{16}(q+1+2s-4t),\\
C=\frac{1}{16}(q+1-6s),\\
D=\frac{1}{16}(q+1+2s+4t),\\
E=\frac{1}{16}(q-3-2s).
\end{array}
\end{equation}

We treat $D=P-N$ and $-D$ as the same type of SDS here. 

\begin{table}
\caption{Fourth-order cyclotomic SDSs}\label{table 3}
\begin{tabular}{|c|c|c|c|c|}
\hline
 $P\cup N$ & $P$  & $k$ & $\lambda$ & \textup{Conditions}\\
 \hline
 \multirow{4}*{$G$} & $G-0_G$  & $q$ & $q-4$  &always hold\\ \cline{2-5}
 & $G-C_i-0_G$  & $q$ & $(q-9)/4$  &$q=4t^2+9$, $t$ is odd \\ \cline{2-5}
 & $C_i+C_j$  & $-$ & $-$  &$-$ \\ \cline{2-5}
 & $C_l$  & $q$ & $(q-1)/4$  &$q=4t^2+1$, $t$ is odd\\ \hline
 \multirow{2}*{$G-0_G$} & $G-C_i-0_G$ & $-$ & $-$  &$-$ \\ \cline{2-5}
 &\multirow{3}*{$C_i+C_j$} & \multirow{3}*{$q-1$} & \multirow{3}*{$-1$}  &$i-j\equiv 2\pmod{4}$ \\ \cline{5-5}
 &&&&$i-j\equiv 1,3\pmod{4}$ \\
 &&&&$p\equiv-1\mod4$\\\hline
 \multirow{13}*{$G-C_j$} & \multirow{6}*{$C_i$}  & \multirow{6}*{$3f+1$} & \multirow{6}*{$(f-1)/4$} & $f\equiv 1\pmod{4}$, $i-j\equiv 1\pmod{4}$,\\
 &&&& $4t+3s=3$\\ \cline{5-5}
 &&&&$f\equiv 1\pmod{4}$, $i-j\equiv 2\pmod{4}$,\\
 &&&& $s=9$\\ \cline{5-5}
 &&&&$f\equiv 1\pmod{4}$, $i-j\equiv 3\pmod{4}$,\\
 &&&& $3s-4t=3$\\\cline{2-5}
 &\multirow{6}*{$G-C_i-C_j-0_G$}  & \multirow{6}*{$3f+1$} & \multirow{6}*{$(f-5)/4$} & $f\equiv 1\pmod{4}$, $i-j\equiv 1\pmod{4}$,\\
 &&&&$3s+4t=-5$\\ \cline{5-5}
 &&&&$f\equiv 1\pmod{4}$, $i-j\equiv 2\pmod{4}$,\\ 
 &&&&$s=-15$\\\cline{5-5}
 &&&&$f\equiv 1\pmod{4}$, $i-j\equiv 3\pmod{4}$,\\ 
 &&&&$4t-3s=5$\\\cline{2-5}
 & $G-C_j-0_G$  & $3f+1$ & $(9f-9)/4$ & $f\equiv 1\pmod {4}$, $s=-7$\\ \hline
 \multirow{6}*{$G-C_j-0_G$} & \multirow{6}*{$C_i$} & \multirow{6}*{$3f$} & \multirow{6}*{$(f-3)/4$} & $f\equiv 3\pmod{4}$, $i-j\equiv1\pmod{4}$\\ 
 &&&&$3s+4t=-1$\\ \cline{5-5}
 &&&&$f\equiv 3\pmod{4}$, $i-j\equiv2\pmod{4}$\\
 &&&&$s=-3$\\ \cline{5-5}
  &&&&$f\equiv 3\pmod{4}$, $i-j\equiv3\pmod{4}$\\
 &&&&$3s-4t=-1$\\ \hline
 $C_i+0_G$ & $C_i$  & $f+1$ & $(f-3)/4$ & $f\equiv 3\pmod{4}$, $s=5$\\ \hline
 $C_i+C_j+0_G$ & $C_i$  & $-$ & $-$ & $-$\\ \hline
 $C_i+C_j$ & $C_i$  & $-$ & $-$ & $-$\\ 
 \hline
 \end{tabular}
 \end{table}

In Table \ref{table 3}, we have listed all the forms, parameters, and conditions required for the existence of fourth-order cyclotomic SDSs. All SDSs in this table, except for the $(q,q,q-4)$-SDS, need to satisfy $q\equiv 1\mod 4$. In the sequel of this section, we will analyze each case presented in Table \ref{table 3} and provide  detailed proofs for the first two cases. And for Cases 3-6, their proofs are tedious verification similar to Cases 1-2, we leave them to readers.
\subsection{Case 1: $P\cup N=G$}
In this case, we have $N=G-P$ and 
\begin{equation}\label{Eqn_Case1D}
DD^{-1}=(|G|-4|P|)G+4PP^{-1}=(|G|-4|N|)G+4NN^{-1}.
\end{equation}
Then $D$ is an SDS if and only if $P$ or $N$ is a DS. Now we focus on the case where $P$ or $N$ is a cyclotomic DS. By the necessary and sufficient conditions for $C_0$ to be a difference set in \cite{Lehmer}, Balmaceda  collected the following results in \cite[Theorem 2.6]{Jose}. 

\begin{lemma}\label{lem_C_0}
Let $\bF_q$ be the finite field of order $q$, where $q$ is a power of an odd prime $p$. Let $C_0$ be the $4$-th order cyclotomic class defined in Eq. ~\ref{Eqn_C0}. 
\begin{itemize}
\item[1)] $C_0$ is a cyclotomic DS in $G=(\bF_q,+)$ with parameters $(q,(q-1)/4,(q-5)/16)$ if and only if $q=4t^2+1$ and $t$ is odd.
\item[2)] $C_0+0_G$ is a cyclotomic DS in $G=(\bF_q,+)$ with parameters $(q,(q+3)/4),(q+3)/16)$ if and only if $q=4t^2+9$ and $t$ is odd.
\end{itemize}
\end{lemma}

\begin{proposition}
Let $q\equiv 1\pmod {4}$ be a prime power and $q=4f+1$. Let $\bF_q$ be a finite field of order $q$ and $G=(\bF_q,+)$. 
\begin{itemize}
    \item[(1)] When $P=G-0_G$ and $N=0_G$, $D=P-N$ is a $(q,q,q-4)$-SDS.
    \item[(2)] When $P=G-C_l-0_G$ and $N=C_l+0_G$ with $l\in\{0,1,2,3\}$, $D=P-N$ is a $(q,q,\frac{q-9}{4})$-SDS if and only if $q=4t^2+9$ and $t$ is odd.
    \item[(3)] When $P=C_i+C_j$ and $N=C_k+C_l+0_G$ with $\{i,j,k,l\}=\{0,1,2,3\}$, $D=P-N$ is not an SDS.
    \item[(4)] When $P=C_l$ and $N=G-C_l$ with $l\in\{0,1,2,3\}$, $D=P-N$ is a $(q,q,\frac{q-1}{4})$-SDS if and only if $q=4t^2+1$ and $t$ is odd.  
\end{itemize} 
\begin{proof}
(1) Let $P=G-0_G$ and $N=G-P$, we compute that 
\[
PP^{-1}=(G-0_G)(G-0_G)^{-1}=(|G|-2)G+0_G.
\]
Then $P$ is a DS with parameters $(v,v-1,v-2)$ and $D$ is a $(v,v,v-4)$-SDS.

(2) Let $P=G-C_l-0_G$ and $N=G-P$ with $l\in\{0,1,2,3\}$. Note that $N$ is a DS if and only if $q=4t^2+9$ with $t$ being odd by Lemma \ref{lem_C_0}. In this case, $N$ has parameters $(q,(q+3)/4,(q+3)/16)$. By Eq. ~\ref{Eqn_Case1D}, we compute that 
\[
\begin{aligned}
DD^{-1}&=(|G|-4|N|)G+4NN^{-1}\\
&=(|G|-4|N|)G+4(\frac{q+3}{16}(G-0_G)+\frac{q+3}{4}0_G)\\
&=\frac{q-9}{4}(G-0_G)+q0_G.
\end{aligned}
\]
Hence $D$ is a $(q,q,\frac{q-9}{4})$-SDS.

(3) Let $P=C_i+C_j$ and $N=G-P$ with $\{i,j\}\subset\{0,1,2,3\}$ and $i\neq j$, we compute that 
\begin{equation}\label{Eqn_1.3}
\begin{aligned}
(C_i+C_j)(C_i+C_j)^{-1}&=C_iC_i^{-1}+C_iC_j^{-1}+C_jC_i^{-1}+C_jC_j^{-1}\\
=&\sum_{k=0}^{e-1}\left[(\frac{q-1}{2},k)+(j-i+\frac{q-1}{2},k)\right]C_{k+i}\\
&+\sum_{k=0}^{e-1}\left[(\frac{q-1}{2},k)+(i-j+\frac{q-1}{2},k)\right]C_{k+j}+2|C_0|0_G\\
=&\sum_{k=0}^{e-1}\left[(\frac{q-1}{2},k)+(j-i+\frac{q-1}{2},k)\right.\\
&\left.+(\frac{q-1}{2},k+i-j)+(i-j+\frac{q-1}{2},k+i-j)\right]C_{k+i}+2|C_0|0_G.
\end{aligned}
\end{equation}

We note that $q=4f+1$,
\begin{equation}
\frac{q-1}{2}=2f\equiv\begin{cases}2&\textup{if }f \textup{ is odd}\\0&\textup{if }f \textup{ is even}\end{cases}\pmod{4}.
\end{equation}

If $f$ is odd,  Eq. \ref{Eqn_1.3} equals 
\[
\sum_{k=0}^{e-1}[(2,k)+(j-i+2,k)+(2,k+i-j)+(i-j+2,k+i-j)]C_{k+i}+2|C_0|0_G.
\]
We are primarily concerned with the coefficients of all $C_{k+i}$ and ensure that the coefficients of $C_{k+i}$ are equal to each other for any $0\leq k\leq e-1$. Next, we will discuss the cases where $j-i$ takes on the values of $1,2$ and $3$ respectively.
\begin{itemize}
    \item[1)] When $j-i=1,3$, the coefficient of $C_{k+i}$ equals to $A_k=(3,k)+(2,k)+(1,k-1)+(2,k-1)$. By Table \ref{table 2}, we have $P$ is a DS if and only if the elements in $\{A_k:\ 0\leq k\leq e-1\}$ are equal to each other, where 
    \[
    \begin{aligned}
    A_k=\begin{cases}
    E+A+B+E & \textup{if }k=0,\\
    D+E+E+A & \textup{if }k=1,\\
    B+A+E+E & \textup{if }k=2,\\
    E+E+D+A & \textup{if }k=3.
    \end{cases}
    \end{aligned}
    \]
    It implies that $t=0$. By the definition of $t$ in Lemma \ref{thm_cctable}, we have $p\equiv -1\pmod{4}$ and $n$ is an even number. In this case, $f$ must be even, which contradicts the assumption. 
    \item[2)] When $j-i=2$, the coefficient of $C_{k+i}$ equals to $A_k=(0,k)+(2,k)+(0,k-2)+(2,k-2)$. By Table \ref{table 2}, we have $P$ is a DS if and only if the elements in $\{A_k:\ 0\leq k\leq e-1\}$ are equal to each other, where
    \[
    \begin{aligned}
    A_k=\begin{cases}
    A+A+C+A & \textup{if }k=0,\\
    B+E+D+E & \textup{if }k=1,\\
    C+A+A+A & \textup{if }k=2,\\
    D+E+B+E & \textup{if }k=3.
    \end{cases}
    \end{aligned}
    \]
    This is equivalent to $3A+C=B+D+2E$ which implies that $-20=-4$. Obviously, we get a contradiction.  
%    \item[3)] $j-i=3$, the coefficient of $C_{k+i}$ equals to $(1,k)+(2,k)+(3,k+1)+(2,k+1)$. By the Table \ref{table 2}, we have $P$ is a DS if and only if 
% % % % %     \[
  %   \begin{aligned}
 % % % %    k=0:&E+A+D+E\\
% % %     k=1:&E+E+B+A\\
 % %    k=2:&D+A+E+E\\
 %   k=3:&B+E+E+A
  %   \end{aligned}
 %    \]
%     are equal to each other. This is equivalent to $B=D$ which implies that $t=0$. 
\end{itemize}

If $f$ is even, Eq. \ref{Eqn_1.3} equals 
\[
\sum_{k=0}^{e-1}[(0,k)+(j-i,k)+(0,k+i-j)+(i-j,k+i-j)]C_{k+i}+2|C_0|0_G.
\]
By  similar arguments, the difference set $P$ does not exist.

(4) Let $P=C_l$ and $N=G-P$ with $l\in\{0,1,2,3\}$. Note that $P$ is a DS if and only if $q=4t^2+1$ with $t$ being odd by Lemma \ref{lem_C_0}. In this case, $P$ has parameters $(q,(q-1)/4,(q-5)/16)$. By Eq. \ref{Eqn_Case1D}, we compute that 
\[
\begin{aligned}
DD^{-1}&=(|G|-4|P|)G+4PP^{-1}\\
&=(|G|-4|P|)G+4(\frac{q-5}{16}(G-0_G)+\frac{q-1}{4}0_G)\\
&=\frac{q-1}{4}(G-0_G)+q0_G.
\end{aligned}
\]
Hence $D$ is a $(q,q,\frac{q-1}{4})$-SDS.
\end{proof}
\end{proposition}

\subsection{Case 2: $P\cup N=G-0_G$} In this case either $P=C_i+C_j+C_k=G-C_l-0_G$ or $P=C_i+C_j=G-C_k-C_l-0_G$ with $\{i,j,k,l\}=\{0,1,2,3\}$. Note that when $i-j\equiv 2\pmod{4}$, $P$ is either the squares or nonsquares set.  By Lemma \ref{sds_pds} and Example \ref{exm_Paley}, we can see that $D=P-N$ is an SDS with parameters $(q,q-1,-1)$.

\begin{proposition}
Let $p$ be an odd prime and $q=p^n\equiv 1\pmod{4}$. Let $\bF_q$ be a finite field of order $q$ and $G=(\bF_q,+)$.
\begin{itemize}
    \item[(1)] When $P=G-C_l-0$ and $N=C_l$ with $l\in\{0,1,2,3\}$, $D=P-N$ is not an SDS.
    \item[(2)] When $P=C_i+C_j$ and $N=C_k+C_l$ with $\{i,j,k,l\}=\{0,1,2,3\}$, $D=P-N$ is a $(q,q-1,-1)$-SDS if and only if $i-j\equiv 2\pmod{4}$ or $i-j\equiv 1,3\pmod{4}$ and $p\equiv -1\pmod{4}$.
\end{itemize}
\begin{proof}
(1) If $P=G-C_l-0$ and $N=C_l$, then we compute that 
\[
\begin{aligned}
(G-2C_l-0)(G-2C_l-0)^{-1}=(|G|-4|C_l|-2)G+4C_lC_l^{-1}+2C_l+2C_l^{-1}+0_G.
\end{aligned}
\]
The cyclotomic class part equals 
\begin{equation}\label{Eqn_2.1}
2C_lC_l^{-1}+C_l+C_l^{-1}=\sum_{k=0}^{e-1}2(\frac{q-1}{2},k)C_{k+l}+C_l+C_{l+\frac{q-1}{2}}+2|C_0|0_G.
\end{equation}

If $f$ is odd, 
\[
\begin{aligned}
\sum_{k=0}^{e-1}2(2,k)C_{k+l}+C_l+C_{l+2}=[2(2,0)+1]C_l+2(2,1)C_{l+1}+[2(2,2)+1]C_{l+2}+2(2,3)C_{l+3}.
\end{aligned}
\]
Since $D$ is an SDS if and only if the coefficients of the above equation with respect to $C_{l+k}$, $0\leq k\leq 3$, are all equal. By Table \ref{table 2}, this implies that $s=-1$. By the definition of $s$ in Theorem \ref{thm_cctable}, it is not possible. Then $f$ must be even and 
\[
\begin{aligned}
\sum_{k=0}^{e-1}2(0,k)C_{k+l}+2C_l=[2(0,0)+2]C_l+2(0,1)C_{l+1}+2(0,2)C_{l+2}+2(0,3)C_{l+3}.
\end{aligned}
\]
By Table \ref{table 1}, $D$ is an SDS if and only if $t=0$ and $s=3$. In this case, $q=t^2+s^2=9$ and $f=1$ which implies a contradiction. 

(2) If $P=C_i+C_j$ and $N=C_k+C_l$ with $\{i,j,k,l\}=\{0,1,2,3\}$, then we compute that 
\begin{equation}\label{Eqn_2.2}
\begin{aligned}
DD^{-1}=&(G-2C_i-2C_j-0_G)(G-2C_i-2C_j-0_G)^{-1}\\
=&(|G|-4|C_i|-4|C_j|-2)G+4C_iC_i^{-1}+4C_iC_j^{-1}+4C_jC_i^{-1}+4C_jC_j^{-1}\\
&+2C_i+2C_j+2C_i^{-1}+2C_j^{-1}+0_G.
\end{aligned}
\end{equation}
In the beginning of this section, we have already discussed the case when $i-j\equiv 2\pmod{4}$. Now we only need to consider the case when $i-j\equiv 1,3\pmod{4}$.

If $f$ is odd and $i-j=1,3$, the cyclotomic class part of Eq. \ref{Eqn_2.2} equals 
\[
\begin{aligned}
&2(C_iC_i^{-1}+C_jC_j^{-1}+C_iC_j^{-1}+C_jC_i^{-1})+C_i+C_j+C_i^{-1}+C_j^{-1}\\
&=2\sum_{k=0}^{e-1}[(2,k)+(j-i+2,k)+(2,k+i-j)\\
&+(i-j+2,k+i-j)]C_{k+i}+C_i+C_j+C_{i+2}+C_{j+2}+2|C_0|0_G,
\end{aligned}
\]
The coefficient of $C_{k+i}$ equals  
\[
       2[(2,k)+(1,k)+(2,k+1)+(3,k+1)]
\]
    and 
\[
       C_i+C_j+C_{i+2}+C_{j+2}=C_i+C_{i+3}+C_{i+2}+C_{i+1}.
\]
By Table \ref{table 2}, $D$ is an SDS if and only if the coefficients of Eq. ~\ref{Eqn_2.2} with respect to $C_{l+k}$, $0\leq k\leq 3$, are all equal. By Table \ref{table 2}, we get $B=D$ which implies that $t=0$. By the definition of $t$ in Lemma \ref{thm_cctable}, we have $p\equiv -1\pmod{4}$ and $n$ is even. In this case, $f$ must be even, which contradicts the assumption.

If $f$ is even and $i-j=1,3$, we use the same method and obtain $t=0$. By the definition of $t$ in Lemma \ref{thm_cctable}, we have $p\equiv -1\pmod{4}$.

So $D$ is an SDS if and only if $i-j\equiv 2\pmod{4}$ or $i-j\equiv 1,3\pmod{4}$ and $p\equiv -1\pmod{4}$. Since $k=|P|+|N|=q-1$ and $|P|-|N|=0$, we get $\lambda=-1$ by Eq. ~\ref{ness1}. 
\end{proof}
\end{proposition}

\subsection{Case 3: $P\cup N=G-C_j$} In this case we have the following three situations.

(1) If $P=C_i$ and $N=G-C_i-C_j$ with $\{i,j\}\subset \{0,1,2,3\}$ and $i\neq j$, then $k=|P|+|N|=3f+1$ and $|P|-|N|=-f-1$. By Eq. ~\ref{ness1}, we get $\lambda=\frac{f-1}{4}$. 

(2) If $P=G-C_i-C_j-0_G$ and $N=C_i+0_G$ with $\{i,j\}\subset \{0,1,2,3\}$ and $i\neq j$, then $k=|P|+|N|=3f+1$ and $|P|-|N|=f-1$. By Eq. ~\ref{ness1}, we get $\lambda=\frac{f-5}{4}$.

(3) If $P=G-C_i-0_G$ and $N=0$ with $i\in \{0,1,2,3\}$, then $k=|P|+|N|=3f+1$ and $|P|-|N|=3f-1$. By Eq. ~\ref{ness1}, we get $\lambda=\frac{9f-9}{4}$. 

In the above three cases, we have $f\equiv 1\pmod{4}$ since $\lambda$ is an integer.

\begin{proposition}
Let $p$ be an odd prime, $q=p^n\equiv 1\pmod{4}$ and $q=4f+1$. Let $G=(\bF_q,+)$ and $w$ be a generator of $\bF_q^*$. If $p\equiv -1\pmod{4}$, let $s=(-p)^{n/2}$ and $t=0$. If $p\equiv1\pmod{4}$, let $s$ be uniquely determined by $q=s^2+t^2$, $p\nmid s$, $s\equiv1\pmod{4}$, and  $t$ be uniquely determined by $w^{(q-1)/4}\equiv s/t\pmod{p}$. 
\begin{itemize}
    \item[(1)] When $P=C_i$ and $N=G-C_i-C_j$ with $\{i,j\}\subset \{0,1,2,3\}$ and $i\neq j$, $D=P-N$ is a $(q,3f+1,\frac{f-1}{4})$ if and only if $f\equiv 1\pmod{4}$ and $s,t$ satisfy one of the conditions: 
    \begin{itemize}
        \item[1)]$i-j\equiv 1\pmod{4}$, $4t+3s=3$; 
        \item[2)]$i-j\equiv 2\pmod{4}$, $s=9$; 
        \item[3)]$i-j\equiv 3\pmod{4}$, $3s-4t=3$.
    \end{itemize}
    \item[(2)] When $P=G-C_i-C_j-0_G$ and $N=C_i+0_G$ with $\{i,j\}\subset \{0,1,2,3\}$ and $i\neq j$, then $D=P-N$ is a $(q,3f+1,\frac{f-5}{4})$ if and only if $f\equiv 1\pmod{4}$ and $s,t$ satisfy one of the conditions: 
    \begin{itemize}
        \item[1)]$i-j\equiv 1\pmod{4}$, $3s+4t=-5$;
        \item[2)]$i-j\equiv 2\pmod{4}$, $s=-15$;
        \item[3)]$i-j\equiv 3\pmod{4}$, $4t-3s=5$.
    \end{itemize}
    \item[(3)] When $P=G-C_i-0_G$ and $N=0_G$ with $i\in \{0,1,2,3\}$, then $D=P-N$ is a $(q,3f+1,\frac{9f-9}{4})$ if and only if $f\equiv 1\pmod{4}$ and $s=-7$. 
\end{itemize}
\end{proposition}

\subsection{Case 4: $P\cup N=C_i+C_j+C_k$} 
Without loss of generality, assume that $P=C_i$ and $N=G-C_i-C_j-0_G$ with $\{i,j\}\in \{0,1,2,3\}$ and $i\neq j$, then $k=|P|+|N|=3f$ and $|P|-|N|=-f$. By Eq. ~\ref{ness1}, we get $\lambda=\frac{f-3}{4}$. Note that $f\equiv 3\pmod{4}$ since $\lambda$ is an integer. 

\begin{proposition}
Let $p$ be an odd prime, $q=p^n\equiv 1\pmod{4}$ and $q=4f+1$. Let $G=(\bF_q,+)$ and $w$ be a generator of $\bF_q^*$. If $p\equiv -1\pmod{4}$, let $s=(-p)^{n/2}$ and $t=0$. If $p\equiv1\pmod{4}$, let $s$ be uniquely determined by $q=s^2+t^2$, $p\nmid s$, $s\equiv1\pmod{4}$, and  $t$ be uniquely determined by $w^{(q-1)/4}\equiv s/t\mod{p}$. If $P=C_i$ and $N=G-C_i-C_j-0_G$ with $\{i,j\}\in \{0,1,2,3\}$ and $i\neq j$, then $D=P-N$ is a $(q,3f,\frac{f-3}{4})$ if and only if $f\equiv 3\pmod{4}$ and $s,t$ satisfy one of the conditions: 
\begin{itemize}
\item[1)] $i-j\equiv 1\pmod{4}$, $3s+4t=-1$;
\item[2)] $i-j\equiv 2\pmod{4}$, $s=-3$; 
\item[3)] $i-j\equiv 3\pmod{4}$, $3s-4t=-1$.
\end{itemize}
\end{proposition}

\subsection{Case 5: $P\cup N=C_i+0_G$} Without loss of generality, assume that $P=C_i$ and $N=0_G$, then we have $k=|P|+|N|=f+1$ and $|P|-|N|=f-1$. By Eq. ~\ref{ness1}, we get $\lambda=\frac{f-3}{4}$. Note that $f\equiv 3\pmod{4}$ since $\lambda$ is an integer.

\begin{proposition}
Let $p$ be an odd prime, $q=p^n\equiv 1\pmod{4}$ and $q=4f+1$. Let $G=(\bF_q,+)$ and $w$ be a generator of $\bF_q^*$. If $p\equiv -1\pmod{4}$, let $s=(-p)^{n/2}$ and $t=0$. If $p\equiv1\pmod{4}$, let $s$ be uniquely determined by $q=s^2+t^2$, $p\nmid s$, $s\equiv1\pmod{4}$, and $t$ be uniquely determined by $w^{(q-1)/4}\equiv s/t\mod{p}$. If $P=C_i$ and $N=0$ with $i\in \{0,1,2,3\}$, then $D=P-N$ is a $(q,q-1,\frac{f-3}{4})$ if and only if $f\equiv 3\pmod{4}$ and $s=5$.
\end{proposition}

\subsection{Case 6: $P\cup N=C_i+C_j+0_G \textup{ or } C_i+C_j$ } For this case $P=C_i$ and $N=C_j+0_G$ or $N=C_j$, both of them will lead to $\lambda=-1/2$, which is a contradiction.

\section{Conclusion}\label{sec_con}
In this paper, we give three constructions for SDSs which are based on PDSs, product methods and cyclotomic classes, respectively. The former two cases provide the infinite classes with parameters $\lambda=-1,1$ and $161$, and the later case provides the infinite classes with parameter $\lambda$ almost equal to $\frac{cv}{4}$ with $c=0,1,4,9$. This fills the gaps in the literature on their existence of these infinite classes.

\section{Acknowledgements}
Gennian Ge was supported by the National Key Research and Development Program of China under Grant 2020YFA0712100 and Grant 2018YFA0704703, the National Natural Science Foundation of China under Grant 11971325 and Grant 12231014, and Beijing Scholars Program. Tingting Chen was supported by the Zhejiang Provincial Natural Science Foundation of China under Grant LQ23A010015.

\bibliographystyle{IEEEtran}
\bibliography{sds}

\begin{thebibliography}{10}
\providecommand{\url}[1]{#1}
\csname url@samestyle\endcsname
\providecommand{\newblock}{\relax}
\providecommand{\bibinfo}[2]{#2}
\providecommand{\BIBentrySTDinterwordspacing}{\spaceskip=0pt\relax}
\providecommand{\BIBentryALTinterwordstretchfactor}{4}
\providecommand{\BIBentryALTinterwordspacing}{\spaceskip=\fontdimen2\font plus
\BIBentryALTinterwordstretchfactor\fontdimen3\font minus
  \fontdimen4\font\relax}
\providecommand{\BIBforeignlanguage}[2]{{%
\expandafter\ifx\csname l@#1\endcsname\relax
\typeout{** WARNING: IEEEtran.bst: No hyphenation pattern has been}%
\typeout{** loaded for the language `#1'. Using the pattern for}%
\typeout{** the default language instead.}%
\else
\language=\csname l@#1\endcsname
\fi
#2}}
\providecommand{\BIBdecl}{\relax}
\BIBdecl

\bibitem{LiuS2018}
Z.~Liu, Y.~L. Guan, U.~Parampalli, and S.~Hu, ``Spectrally-constrained
  sequences: Bounds and constructions,'' \emph{IEEE Transactions on Information
  Theory}, vol.~64, no.~4, pp. 2571--2582, 2018.

\bibitem{Cusick2015}
T.~Cusick, C.~Ding, and A.~Renvall, \emph{Stream ciphers and number theory,
  revised edition}.\hskip 1em plus 0.5em minus 0.4em\relax North Holland,
  Amsterdam, 2005.

\bibitem{Ding}
C.~Ding, \emph{Codes from difference sets}.\hskip 1em plus 0.5em minus
  0.4em\relax World Scientific, 2014.

\bibitem{Tsai2011}
L.-S. Tsai, W.-H. Chung, and D.-s. Shiu, ``Synthesizing low autocorrelation and
  low {PAPR OFDM} sequences under spectral constraints through convex
  optimization and {GS} algorithm,'' \emph{IEEE Transactions on Signal
  Processing}, vol.~59, no.~5, pp. 2234--2243, 2011.

\bibitem{YeZ2022}
Z.~Ye, Z.~Zhou, Z.~Liu, X.~Tang, and P.~Fan, ``New spectrally constrained
  sequence sets with optimal periodic cross-correlation,'' \emph{IEEE
  Transactions on Information Theory}, vol.~69, no.~1, pp. 610--625, 2022.

\bibitem{Gordon}
\BIBentryALTinterwordspacing
D.~M. Gordon, ``Signed difference sets,'' \emph{Designs, Codes and
  Cryptography}, vol.~91, no.~5, pp. 2107--2115, 2023. [Online]. Available:
  \url{https://doi.org/10.1007/s10623-022-01171-8}
\BIBentrySTDinterwordspacing

\bibitem{Tao}
R.~Tao, T.~Feng, and W.~Li, ``A construction of minimal linear codes from
  partial difference sets,'' \emph{IEEE Transactions on Information Theory},
  vol.~67, no.~6, pp. 3724--3734, 2021.

\bibitem{Ye}
Z.~Ye, Z.~Zhou, S.~Zhang, and X.~Tang, ``Optimal ternary sequence sets
  rigorously achieving the {W}elch bound (in chinese),'' \emph{Sci Sin Math},
  vol.~53, pp. 1--12, 2023.

\bibitem{Hall}
M.~Hall, ``A survey of difference sets,'' \emph{Proceedings of the American
  Mathematical Society}, vol.~7, no.~6, pp. 975--986, 1956.

\bibitem{S.L.MA}
S.~L. Ma, ``A survey of partial difference sets,'' \emph{Designs, Codes and
  Cryptography}, vol.~4, no.~4, pp. 221--261, 1994.

\bibitem{ArasuKT2}
K.~Arasu and J.~R. Hollon, ``Group developed weighing matrices,''
  \emph{Australas. J Comb.}, vol.~55, pp. 205--234, 2013.

\bibitem{Tan}
M.~M. Tan, ``Group invariant weighing matrices,'' \emph{Designs, Codes and
  Cryptography}, vol.~86, pp. 2677--2702, 2018.

\bibitem{LeKH}
K.~H. Leung and B.~Schmidt, ``Structure of group invariant weighing matrices of
  small weight,'' \emph{Journal of Combinatorial Theory, Series A}, vol. 154,
  pp. 114--128, 2018.

\bibitem{D.M}
D.~M. Gordon, ``{La Jolla Combinatorics Repository},''
  \url{https://www.dmgordon.org}, 2022.

\bibitem{A.E}
A.~E. Brouwer, ``Parameters of strongly regular graphs,''
  \url{https://www.win.tue.nl/~aeb/graphs/srg/srgtab.html}, 2019.

\bibitem{Helleseth}
T.~Helleseth and G.~Gong, ``New nonbinary sequences with ideal two-level
  autocorrelation,'' \emph{IEEE Transactions on Information Theory}, vol.~48,
  no.~11, pp. 2868--2872, 2002.

\bibitem{Xiang}
Q.~Xiang, ``On balanced binary sequences with two-level autocorrelation
  functions,'' \emph{IEEE Transactions on Information Theory}, vol.~44, no.~7,
  pp. 3153--3156, 1998.

\bibitem{Wang}
\BIBentryALTinterwordspacing
Z.~Wang and G.~Gong, ``Constructions of complementary sequence sets and
  complete complementary codes by ideal two-level autocorrelation sequences and
  permutation polynomials,'' \emph{IEEE Transactions on Information Theory},
  2023. [Online]. Available: \url{http://dx.doi.org/10.1109/TIT.2023.3258180}
\BIBentrySTDinterwordspacing

\bibitem{Jose}
J.~M.~P. Balmaceda and B.~M. Estrella, ``Difference sets from unions of
  cyclotomic classes of orders $12$, $20$, and $24$,'' \emph{Philippine Journal
  of Science}, vol. 150, no.~6B, pp. 1803--1810, 2021.

\bibitem{ArasuKT}
K.~Arasu, D.~Jungnickel, S.~L. Ma, and A.~Pott, ``Strongly regular {C}ayley
  graphs with $\lambda- \mu=-1$,'' \emph{Journal of Combinatorial Theory,
  Series A}, vol.~67, no.~1, pp. 116--125, 1994.

\bibitem{ArasuKT1}
K.~Arasu, D.~Jungnickel, and A.~Pott, ``Symmetric divisible designs with
  $k-\lambda_1= 1$,'' \emph{Discrete mathematics}, vol.~97, no. 1-3, pp.
  25--38, 1991.

\bibitem{ChanYK}
Y.-K. Chan, M.-K. Siu, and P.~Tong, ``Two-dimensional binary arrays with good
  autocorrelation,'' \emph{Information and Control}, vol.~42, no.~2, pp.
  125--130, 1979.

\bibitem{S.L.MA1}
\BIBentryALTinterwordspacing
S.~Ma, ``Partial difference sets,'' \emph{Discrete Mathematics}, vol.~52,
  no.~1, pp. 75--89, 1984. [Online]. Available:
  \url{https://www.sciencedirect.com/science/article/pii/0012365X84901055}
\BIBentrySTDinterwordspacing

\bibitem{Polhill}
J.~Polhill, ``Paley partial difference sets in groups of order $n^4$ and $9n^4$
  for any odd $n>1$,'' \emph{Journal of Combinatorial Theory, Series A}, vol.
  117, no.~8, pp. 1027--1036, 2010.

\bibitem{WangZ}
Z.~Wang, ``Paley type partial difference sets in abelian groups,''
  \emph{Journal of Combinatorial Designs}, vol.~28, no.~2, pp. 149--152, 2020.

\bibitem{Golay}
M.~J. Golay, ``Notes on digital coding,'' \emph{Proc. IEEE}, vol.~37, p. 657,
  1949.

\bibitem{Paley}
R.~E. Paley, ``On orthogonal matrices,'' \emph{Journal of mathematics and
  Physics}, vol.~12, no. 1-4, pp. 311--320, 1933.

\bibitem{Katre}
S.~Katre and A.~Rajwade, ``Resolution of the sign ambiguity in the
  determination of the cyclotomic numbers of order $4$ and the corresponding
  {J}acobsthal sum,'' \emph{Mathematica Scandinavica}, vol.~60, pp. 52--62,
  1987.

\bibitem{Lehmer}
E.~Lehmer, ``On residue difference sets,'' \emph{Canadian Journal of
  Mathematics}, vol.~5, pp. 425--432, 1953.

\end{thebibliography}
\end{document}